\documentclass[a4paper,11pt]{amsart}
\usepackage[utf8]{inputenc}
\usepackage{datetime}
\usepackage[draft]{todonotes}

\usepackage{amssymb}
\usepackage[bookmarks=false]{hyperref}
\usepackage{cleveref}
\usepackage{fancyhdr}
\usepackage[shortlabels]{enumitem}

\newtheorem{Lem}{Lemma}

\newtheorem{lemma}[Lem]{Lemma}

\newtheorem{theorem}[Lem]{Theorem}

\theoremstyle{definition}

\DeclareMathOperator{\re}{Re}
\DeclareMathOperator{\I}{I}
\DeclareMathOperator{\II}{II}

\newcommand{\N}{\mathbb{N}}
\newcommand{\C}{\mathbb{C}}

\newcommand{\restrict}{\!\upharpoonright\!}

\DeclareMathOperator{\fin}{fin}
\DeclareMathOperator{\supp}{supp}
\newcommand{\gf}{G^{\fin}}

\renewcommand{\epsilon}{\varepsilon}
\renewcommand{\phi}{\varphi}

\newcommand{\myspace}[1]{\vspace{1ex}}

\newcommand{\myright}{\thepage}
\newcommand{\myleft}{Fabio E.~Tonti, Asger Törnquist}
\fancypagestyle{plain}{%
\fancyhf{} 
\fancyhead[L]{\myleft}
\fancyhead[R]{\myright}
}

\date{\today, \currenttime}
\title{A short proof of Thoma's theorem on type I groups}

\author{Fabio Elio Tonti}
\address{Institute of Discrete Mathematics and Geometry, 
Technische Universit\"at Wien, Wiedner Hauptstrasse 8-10/104, 1040 Vienna, Austria}
\email{fabio.tonti@tuwien.ac.at}

\author{Asger T\"ornquist}
\address{Department of Mathematical Sciences, University of Copenhagen, Universitetsparken 5, 2100 Copenhagen, Denmark}
\email[Corresponding author]{asgert@math.ku.dk}

\begin{document}
\begin{abstract}
In the theory of unitary group representations, a group is called \emph{type I} if all factor representations are of type I, and by a celebrated theorem of James Glimm \cite{Glimm61}, the type I groups are precisely those groups for which the irreducible unitary representations are what descriptive set theorists now call ``concretely classifiable''. Elmar Thoma \cite{Thoma64} proved the following surprising characterization of the countable discrete groups of type I: They are precisely those that contain a finite index abelian subgroup. In this paper we give a new, simpler proof of Thoma's theorem, which relies only on relatively elementary methods.
\end{abstract}

\maketitle


\section{Introduction}\label{ss:intro}

{\bf (A)} In recent years there has  been a renewed interest in the classification problem for irreducible unitary representations. The current interest comes primarily from descriptive set theorists, who in the past 30 years have developed a vast theory, called \emph{Borel reducibility theory}, for measuring the relative complexity of classification problems in mathematics. Turning these techniques toward the problem of classifying irreducible representations of groups (and C$^*$-algebras) is, satisfyingly, a sort of homecoming: Modern Borel reducibility theory grew out of the smooth/non-smooth dichotomy that originated in the work of Mackey, Glimm and Effros, among others, exactly with the purpose of quantifying the phenomenon that for some groups and C$^*$-algebras, their irreducible unitary representations are concretely classifiable (i.e.,~the unitary dual is ``smooth''), while for other groups and C$^*$-algebras, no reasonable classification seems to be possible (i.e.,~the unitary dual is not smooth, equivalently, is not countably separated). See \cite{Glimm61, Glimm61+, Effros65, Effros81}.

Glimm's celebrated ``smooth dual if and only if type I'' theorem (in (\cite[Theorem 2]{Glimm61}) is a wellspring: A string of results in descriptive set theory,  \cite{Hjorth97, Farah12, Thomas15}, have sharpened Glimm's result considerably, showing that when the unitary dual is not smooth, then the complexity of classifying the irreducible representations is very high. It is still an active research problem to determine the exact complexity of classifying the irreducible unitary representations for various groups (see \cite{Thomas15}).

In the important special case when we consider unitary representations of countable discrete groups, a cornerstone theorem is the following result due to Thoma:

\begin{theorem}[Thoma 1964, {\cite[Satz 6, p.133ff.]{Thoma64}}]\label{t.thoma}
A countable discrete group is type I if and only if it is abelian-by-finite (i.e.,~contains a finite index abelian subgroup).
\end{theorem}

The purpose of this paper is to present a new and more accessible proof of \Cref{t.thoma}. The concrete and easy to understand statement of this theorem makes it broadly useful, not just for unitary representations of groups. For instance, it is often used in ergodic theory in combination with the so-called Gaussian construction to achieve a large family of non-conjugate measure-preserving ergodic actions of a given non-type I group, see e.g. \cite{Hjorth05, Tornquist09, Ioana11, Epstein07}.

\Cref{t.thoma} gives a dichotomy in the classification of irreducible unitary representations: A countable discrete group is either abelian-by-finite, and the classification of its irreducible unitary representations is smooth, or it is not abelian-by-finite, and the classification of its irreducible unitary representations is horribly difficult, as measured by Borel reducibility. It is interesting to note that Thomas \cite{Thomas15} has proved that for all \emph{amenable} non-abelian-by-finite countable discrete groups, the classification of the irreducible unitary representations always form the same non-smooth Borel reducibility degree, independently of the group. Thomas has asked in \cite{Thomas15} if the complexity of classifying the irreducible unitary representations of \emph{non-amenable} countable discrete groups is strictly higher, but this is not yet known. A theorem of this nature could be viewed as giving a second dichotomy for the classification of irreducible unitary representations, with \Cref{t.thoma} being the first dichotomy.

\medskip

{\bf (B)} Let us briefly comment on our proof in relation to Thoma's original proof. The difficult part of the theorem is to show ``only if''. Let us fix a countable discrete group $G$ for this discussion.

Thoma's proof is based on the development of a ``bespoke'', or ``custom made'', direct integral decomposition theory for what he calls ``traces'' on the algebra $\mathfrak A(G)$ of all finitely supported complex valued functions on $G$. These traces in turn correspond to positive definite functions on $G$ that are invariant under conjugacy (``class finite positive definite functions'', in Thoma's terminology).

By contrast, our proof uses only the standard direct integral decomposition theory for von Neumann algebras into factors, and unitary representations into factor representations, and so our proof avoids building up a special, custom made direct integral decomposition theory altogether.

\medskip

In our proof, as well as in Thoma's proof, a key role is played by the normal subgroup 
$$
\gf=\{g\in G: |[g]_G|<\infty\}
$$
consisting of all the elements in $G$ with finite conjugacy classes. Note: Throughout the paper we will use the notation $[g]_G=\{hgh^{-1}: h\in G\}$.

\medskip

The ``easy'' case is when $[G:\gf]=\infty$, i.e.,~$\gf$ has infinite index in $G$. In this case, our proof and Thoma's follow a rather similar route: We show in \Cref{s.easy} that almost every factor representation in the factor decomposition of the  regular representation is of type $\II_1$. Our proof is essentially an elaboration on the well-known proof that the group von Neumann algebra $L(G)$ of an icc group is a $\II_1$ factor. Thoma achieves the same conclusion in \cite[Satz 3]{Thoma64}, but he proves it by analyzing extensions of class finite positive definite functions from $\gf$ to all of $G$ from the point of view of his custom made direct integral theory.

\medskip

The ``hard'' case for us, as well as for Thoma, is the case when $[G:\gf]<\infty$, and it is also here that our proof differs most significantly from Thoma's. Our proof in this case is found in \Cref{s.hard} below. Aiming to prove the contrapositve of the ``only if'' in \Cref{t.thoma}, we assume that $G$ is \emph{not} abelian-by-finite, and proceed to construct inside of $\gf$ a sequence $(G_i)_{i\in\N}$ of pairwise commuting non-abelian subgroups. We show that the subgroups $G_1\cdots G_N$ give rise to von Neumann subalgebras $S(G_1\cdots G_N)$ of $L(G)$ with the property that at least half the factors in the factor decomposition of $S(G_1\cdots G_N)$ will have dimension $>k$ when $N$ is chosen sufficiently large. This in turn gives us that $L(G)$ contains a subalgebra which is of type $\II_1$, and so $L(G)$ is not purely of type I, which implies the result.

Thoma takes quite a different path in the ``hard'' case. He assumes that $G$ is type I, and assumes w.l.o.g.~that $G=\gf$. If $G$ is finitely generated, the assumption $G=\gf$ easily gives that $G$ is abelian-by-finite. Therefore Thoma can assume that $G$ is not finitely generated, and so there is a strictly increasing sequence of finitely generated subgroups $\mathfrak K_0\subset\mathfrak K_1\subset\cdots\subset G$, which exhausts $G$. He then considers an ``extremal'' class finite positive definite function $\alpha$ on $G$, which by his bespoke direct integral representation theory can be identified with a factor representation. He then analyses $\alpha$ in terms of its restrictions $\alpha_i$ to the subgroups $\mathfrak K_i$. The assumption that $G$ is type I puts a bound on the growth of the dimensions of the factor representations that arise from the $\alpha_i$. This bound, through several further arguments, allows Thoma to find a finite index abelian subgroup in $G$.

\medskip

{\bf (C)} We have taken pains to write this paper in such a way that it will be accessible to a broad group of mathematicians, including, we hope, descriptive set theorists and ergodic theorists. Our proof of \Cref{t.thoma} relies heavily on the elementary theory of factor decompositions of unitary representations and von Neumann algebras, and in \Cref{s.bg} we give a fairly detailed account of these matters. The reader who is an expert in von Neumann algebras and/or the theory of unitary representations will probably feel that some of the minor lemmata and their proofs, which we spell out in detail in this paper, are routine and could be omitted. We hope the \emph{expert} will tolerate this level of detail, knowing that these details are spelled out for the benefit of the \emph{non-expert}.

\subsection*{Acknowledgments} The first author is currently a PhD student at the Institute of Discrete Mathematics and Geometry, TU Wien, and a DOC Fellow of the Austrian Academy of Sciences. The first author furthermore gratefully acknowledges support by the Austrian Science Fund (FWF) grants no. P26737 and P30666. The second author thanks the Danish Council for Independent Research for generous support through grant no.\ 7014-00145B. The second author also acknowledges his association with the Centre for Symmetry and Deformation, funded by the Danish National Research Foundation (DNRF92).

\section{Background and notation}\label{s.bg}
%

In this section we review some of the basic decomposition theory of unitary representations and von Neumann algebras and some facts about group von Neumann algebras. A more detailed overview of these things can be found in \cite[III.5, III.3.3]{Blackadar06}, and a fuller account may be found in \cite{Dixmier96, Nielsen80, KR97}. We also fix notation that will be used in the rest of the paper.

\subsection{Unitary representations and von Neumann algebras} 

Given a Hilbert space $\mathcal H$, we will denote by $\mathcal B(\mathcal H)$ the set of bounded operators on $\mathcal H$, and by $\mathcal U(\mathcal H)$ the group of unitary operators on $\mathcal H.$ A \emph{unitary representation} of a countable discrete group $G$ is simply a homomorphism $\pi:G \to \mathcal{U}(\mathcal{H})$. 
To $\pi$ we associate
$$
\pi(G)'=\{T\in\mathcal B(H): (\forall g\in G)\ T\pi(g)=\pi(g) T\}.
$$
It is easy to show that $\pi(G)'$ is a \emph{von Neumann algebra}, i.e.,~a $*$-subalgebra of $\mathcal B(H)$ which is closed in the weak operator topology (equivalently, in the ultraweak operator topology), and which contains the identity operator $I$ on $\mathcal H$. The representation $\pi$ is called a \emph{factor representation} (or sometimes in older references, a \emph{primary} representation) if $\pi(G)'$ is a factor\footnote{Unlike \cite{Blackadar06}, but similar to \cite{Dixmier96,Nielsen80,HT12}, we take the view that it is the structure of $\pi(G)'$ that is the most relevant to analyzing the representation $\pi$, and not $\pi(G)''$. This is because the structure of projections in $\pi(G)'$ is more relevant to analyzing $\pi$ than the projections in $\pi(G)''$, since the projections in $\pi(G)'$ correspond to exactly invariant subspaces of $\pi$.} in the sense of von Neumann algebra theory, the definition of which is: a von Neumann algebra $M\subseteq \mathcal B(H)$ is a \emph{factor} if the \emph{center} of $M$,
$$
Z(M)=\{x\in M: (\forall y\in M)\ xy=yx\},
$$
consist of multiples of $I$, that is $Z(M)=\C I$. Von Neumann factors are naturally categorized in \emph{types}, called I, II, and III, which break into further subtypes. The type of a factor representation $\pi$ is the type of $\pi(G)'$. In this paper we will only need to consider type $I_n$, for $n\in\N$, and $\II_1$ factors, and we give a working definition of these at the end of this section.

Every von Neumann algebra $M\subseteq\mathcal B(\mathcal H)$ admits a direct integral decomposition into factors. This means: There is a measure space $(X,\mu)$ and a measurable assignment $x\mapsto \mathcal H^x$ of Hilbert spaces, and a measurable assignment $x\mapsto F^x\subseteq\mathcal B(\mathcal H^x)$ of factors such that
$$
M=\int F^xd\mu(x),
$$
in the sense that, in a natural way, $\mathcal H=\int\mathcal H^xd\mu(x)$, and $M$ is generated by the measurable functions $f:X\to \mathcal B(\mathcal H^x)$ where $f(x)\in F^x$ for $\mu$-almost every $x$. In this picture, $Z(M)$ corresponds to the measurable functions $f:X\to\mathcal B(\mathcal H^x)$ where $f(x)=c^x I(x)$, where $c^x\in\C$ and $I(x)$ is the identity operator on $\mathcal H^x$.

We note that in the special case when $\mathcal H$ is separable (which will always be the case in this paper), the measure space $(X,\mu)$ in the decomposition is a standard measure space (in the sense of e.g.~\cite{Kechris95}), and we could have replaced ``measurable'' with ``Borel'' above. See \cite{Effros65b,Effros66,HaaWin98,HaaWin00} for a detailed development of the Borel theory of separably acting von Neumann algebras.

\medskip

Turning our attention back to unitary representations, if we are given $\pi:G\to\mathcal U(\mathcal H)$, the factor decomposition of $\pi(G)'$ gives rise to an \emph{integral decomposition of the representation $\pi$ into factor representations}. That is, for any unitary representation, we can find a measure space $(X,\mu)$ and a decomposition $\mathcal H=\int\mathcal H^xd\mu(x)$ of the Hilbert space, and a measurable assignment $x\mapsto \pi^x(g)\in\mathcal U(\mathcal H^x)$ for each $g\in G$, such that $g\mapsto \pi^x(g)$ is a factor representation on $\mathcal H^x$ (a.e).

\subsection{The regular representation and the group von Neumann algebra} 

The left regular representation $\lambda$ is the representation $\lambda:G\to l^2(G)$ given by $\lambda(g)=u_g$, where
$$
(u_g(f))(h)=f(g^{-1}h).
$$
The right regular representation $\rho: G\to l^2(G)$ is defined by $(\rho(g))(f)(h)=f(hg)$.

The (left) group von Neumann algebra $L(G)$ is the von Neumann algebra generated by $\lambda(G)$. We let $R(G)$ be the von Neumann algebra generated by $\rho(G)$. It turns out that $L(G)'=R(G)$ and $R(G)'=L(G)$ (\cite[Theorem 6.7.2]{KR97}). Since much of this paper is be concerned with analyzing the factor decomposition of $L(G)$, it is worth noting that because $L(G)=\rho(G)'$, this actually corresponds to analyzing the factor decomposition of the \emph{right} regular representation. Of course, we could have chosen to focus on $R(G)$ and the representation $\lambda$ instead, with much the same result.

The group von Neumann algebra $L(G)$ is equipped with a natural \emph{trace} $\tau:L(G)\to\C$, which is continuous w.r.t.~the ultraweak operator topology (see \cite[p. 14]{Blackadar06}), and which is uniquely determined by requiring that
$$
\tau(\sum_{g \in G} \alpha_g u_g)=\alpha_e
$$
on all finite sums $\sum_{g \in G} \alpha_g u_g$, where $\alpha_g\in\C$. The trace is a linear functional, and satisfies $\tau(I)=1$ and $\tau(ab)=\tau(ba)$ for all $a,b\in L(G)$. Moreover, the trace is \emph{faithful}, meaning that $\tau(aa^*)=0$ iff $a=0$. In particular, if $p\in L(G)$ is a projection, i.e.,~$p=p^*=p^2$, then $\tau(p)=0$ iff $p=0$. The trace gives rise to an inner product $\langle a,b\rangle_\tau=\tau(ab^*)$ on $L(G)$, which in turn gives rise to a norm $\|a\|_\tau=\sqrt{\langle a,a\rangle_\tau}$. Note that unitaries in $L(G)$ are unit vectors in $\|\cdot\|_\tau$, and that the family $(u_g)_{g\in G}$ forms an orthonormal set in $(L(G),\|\cdot\|_\tau)$. Warning: $L(G)$ is in general not complete in this norm.

\medskip

In our proofs, we will mostly be considering a fixed group $G$ together with a subgroup $H<G$. In this setting, we will denote by $S(H)$ the von Neumann algebra generated by the unitaries $u_h$, where $h$ ranges over $H$. The ambient group $G$ will always be clear from context. Clearly, $S(H)$ is a von Neumann subalgebra of $L(G)$, and the restriction of the trace $\tau$ on $L(G)$ is a trace on $S(H)$.

\medskip
  
We will often consider the factor decomposition of $L(G)$ or $S(H)$ for some $H<G$, e.g. 
$$ 
L(G) = \int_{(X,\mu)} F^x d \mu(x) \, .
$$
Every central projection (i.e.,~projection in $Z(L(G))$) is of the form $p=\int_B I(x) d\mu(x)$ for some measurable $B\subseteq X$, and where $I(x)$ is the identity operator in $F^x$. So we can think of central projections as measurable subsets of $X$, and we let $\supp(p)$ denote the measurable set corresponding to the central projection $p$. We can then define a new measure on $X$ by $\mu_\tau(B)=\tau(p)$, and since $\tau$ is faithful, this measure is absolutely equivalent to $\mu$. For this reason, we could just as well have used the measure $\mu_\tau$ coming from the trace in the factor decomposition of $L(G)$ (or $S(H)$), and this what we will always do from now on when decomposing $L(G)$ and $S(H)$.

When we use the measure coming from the trace in decomposing $L(G)$ or $S(H)$, then the trace itself decomposes nicely: We can find a measurable assignment $x\mapsto \tau^x$ of traces on each factor $F^x$ in the decomposition so that
$$
\tau(a)=\int \tau^x(a^x)d\mu_\tau(x).
$$

\subsection{Type $\I_n$ and $\II_1$} We make the following practical working definition, which will suffice for the purposes of this paper. A von Neumann factor $F$ is type $\I_n$ if it is isomorphic to the matrix algebra $M_n(\C)$. A von Neumann factor $F$ is type $\II_1$ if it is infinite dimensional (as a vector space) and admits an ultraweakly continuous trace $\tau:F\to\C$ with $\tau(I)=1$. See \cite[III.1]{Blackadar06}.

\subsection{Murray-von Neumann equivalence of projections} On one occasion in the proof we will make use of the following notion: In a von Neumann algebra $M$, two projections $p,q\in M$ are \emph{Murray-von Neumann equivalent} if there is $u\in M$ (called a \emph{partial unitary}) such that
$$ 
u^* u=p \text{ and } uu^* =q \, .
$$

\subsection{Corners vs.~subalgebras} Unlike C$^*$-algebras, von Neumann algebras are always assumed to contain the identity operator (which is then the unit element of the algebra). In particular, a von Neumann subalgebra of another von Neumann algebra is always a \emph{unital} subalgebra.

On a few occasions, we will consider the algebra $pMp=pM=Mp$, where $M$ is a von Neumann algebra and $p\in Z(M)$ is a central projection. It is clear that this is a $*$-subalgebra, but it is not unital unless $p=I$. We will call such a subalgebra a \emph{corner} of $M$.

\subsection{Inclusions among abelian von Neumann algebras} Finally, we will on several occasions consider a situation where we are given abelian von Neumann subalgebras $A\subseteq B\subseteq L(G)$, where $G$ is some countable discrete group.

In general, abelian von Neumann algebras are isomorphic to $L^\infty(X,\mu)$ acting by pointwise multiplication in $L^2(X,\mu)$, for some measure space $(X,\mu)$ (see \cite[p.~236]{Blackadar06}). If we identify $A$ (resp.~$B$) with $L^\infty(X_A,\mu_A)$ (resp.~$L^\infty(X_B,\mu_B)$) in this way, we can always assume that the measure comes from the trace inherited from $L(G)$. Moreover, under this identification, the projections in $A$ (resp.~$B$), correspond to measurable subsets of $X_A$ (resp.~$X_B$), and vice versa. Since every projection in $A$ is also a projection in $B$, this means that the measurable subsets of $X_A$ can be identified with a sub-$\sigma$-algebra of the measurable subsets of $X_B$, and since the measures $\mu_A$ and $\mu_B$ are derived from the trace on $L(G)$, they agree on this sub-$\sigma$-algebra.

Another possible view to take is that the situation $A\subseteq B\subseteq L(G)$ gives rise to a measure-preserving surjection $\vartheta: X_B\to X_A$. However, we will work with the sub-$\sigma$-algebra view described above, as it is easier in our setting.


\section{The ``easy'' case: $[G:\gf]=\infty$}\label{s.easy}
In the case when $\gf$ has infinite index in $G$, \Cref{t.thoma} follows from:

\begin{theorem}\label{t.caseinf}
Let $G$ be a countable discrete group, and suppose $[G:\gf]=\infty$. Then every factor in the factor decomposition of $L(G)$ is of type $\II_1$. 
\end{theorem}

To prove this, we need the following two lemmata.

\begin{lemma}\label{l.gf}
\leavevmode
\makeatletter
\@nobreaktrue
\makeatother

1) $Z(L(G))\subseteq S(\gf)$.

2) For any projection $p\in S(\gf)$ and $h\in G\setminus \gf$ we have
$$
\langle p,u_h\rangle_\tau=0.
$$
\end{lemma}

\begin{proof}
\leavevmode
\makeatletter
\@nobreaktrue
\makeatother

1) We omit the proof, as this is \emph{exactly} what the standard proof (see e.g.~\cite{KR97}[Theorem 6.7.5]) of ``$L(G)$ is a factor when $G$ has infinite conjugacy classes (icc)'' actually shows.

2) By definition, $p\in S(\gf)$ can be approximated arbitrarily well in the ultraweak topology by finite sums $\sum_{g \in \gf} \alpha_g u_g$ for a sequence $\alpha_g \in \mathbb{C}$. So for $h \notin \gf$, the product $p u_{h^{-1}}$ is the ultraweak limit of finite sums of the form 
$$
su_{h^{-1}}=\sum_{g \in \gf} \alpha_g u_g u_{h^{-1}}=\sum_{g \in \gf} \alpha_g u_{gh^{-1}},
$$
Since none of these sums contain $u_{e}$ we get $\langle s,u_h\rangle_\tau=\tau(su_{h^{-1}})=0$, and so by continuity $\langle p,u_h\rangle_\tau=0$.
\end{proof}


\begin{lemma}\label{l.orth}
Let
  \[ L(G) = \int_{X} F^x  d  \mu_{\tau}(x) \]
be the factor decomposition of $L(G)$ (w.r.t.~the measure $\mu_\tau$ arising from the trace), let $\tau^x$ be the trace on $F^x$ (as in \Cref{s.bg}), and for each $g\in G$ decompose $u_g$ as
$$
u_g=\int u_g^x  d  \mu_\tau(x) \, ,
$$
where $u_g^x$ is a unitary in $F^x$. Let $g,h\in G$ with $g\notin h\gf$. Then
$$
\langle u_g^x,u_h^x\rangle_{\tau^x}=0
$$
for almost all $x$.
\end{lemma}


\begin{proof}
Let
\[ 
A=\{x \in X : \langle u_g^x, u_h^x\rangle_{\tau^x} \neq 0 \} \, . 
\]
We will show that $ \mu_\tau(A)=0$.

Suppose instead $\mu_\tau(A)>0$. Then for some $\delta>0$ the set
$$
A_\delta=\{x\in X:  \re(\langle u_g^x, u_h^x\rangle_{\tau^x}) >\delta\} \
$$
has positive measure, and so there is $r_0\in\mathbb{R}$ with $r_0>\delta$ such that the set
\[ B = \{x \in X : \lvert \re (\langle u_{g}^x, u_{h}^x \rangle_{\tau^x})  - r_0 \rvert < \frac \delta 2 \} \]
has positive $\mu_\tau$-measure. In $L(G)$, the set $B$ corresponds to the projection 
$$
p=\int_B I(x)d\mu_\tau(x)
$$
where $I(x)$ is the identity in $F^x$. As
 $p \in Z(L(G))$, by Lemma \ref{l.gf} we have $\langle p,u_{g^{-1}h}\rangle=0$. However, the definition of $B$ gives
$$
\re( \int_{B} \langle u_{g}^x, u_{h}^x \rangle_{\tau^x}  d  \mu_\tau(x))=\int_{B} \re(\langle u_{g}^x, u_{h}^x \rangle_{\tau^x})  d  \mu_\tau(x)>\frac\delta 2\mu_\tau(B),
$$
and so since
$$
\re(\langle p, u_{g^{-1}h} \rangle_\tau) =\re(\langle u_{g} p , u_{h} \rangle_\tau) =\re( \int_{B} \langle u_{g}^x, u_{h}^x \rangle_{\tau^x}  d  \mu_\tau(x))
$$
we get that $\langle p, u_{g^{-1}h} \rangle_\tau\neq 0$, which is a contradiction.
\end{proof}

\begin{proof}[Proof of \Cref{t.caseinf}]
Let $G$ be a countable discrete group, and suppose $[G:\gf]=\infty$. Let $(g_i)_{i \in \N}$ be a sequence in $G$ such that $g_i\gf$ enumerates all left cosets injectively. Let $F^x$ and $u_g^x$ be as in the statement of the previous lemma. The previous lemma gives that the unitaries $u_{g_i}^x \in F^x$ are orthogonal w.r.t.~$\langle\cdot,\cdot\rangle_{\tau_x}$ for almost all $x$. Therefore, almost every factor in the decomposition of $L(G)$ is infinite-dimensional, and thus of type II$_1$.
\end{proof}


\section{The ``hard'' case: $[G:\gf]<\infty$}\label{s.hard}

Throughout this section, $G$ denotes a countable discrete group. In the case when $\gf$ has finite index in $G$, \Cref{t.thoma} follows immediately from (B) of the following theorem.

\begin{theorem}\label{t.hard}
Let $G$ be a countable discrete group. 

{\rm (A)} Suppose $G$ contains a sequence of non-abelian pairwise commuting subgroups $(G_i)_{i\in\N}$ (i.e.,~ each $G_i$ is non-abelian, but for $i\neq j$ every element of $G_i$ commutes with every element of $G_j$). Then $L(G)$ is not of type I.

{\rm (B)} If $[G:\gf]<\infty$ and $G$ is not abelian-by-finite, then $G$ contains a sequence of non-abelian pairwise commuting subgroups $(G_i)_{i\in\N}$. Therefore, by {\rm (A)}, $L(G)$ is not type I, and so its factor decomposition contains a $\II_1$ factor.
\end{theorem}

\subsection{Proof of Part (A) of \Cref{t.hard}}

The proof is developed in a sequence of lemmata.

\begin{lemma}\label{l.nonab}
Let $H<G$ be a subgroup, and let
$$
S(H)=\int F^xd\mu_\tau(x)
$$
be the factor decomposition of $S(H)$. Suppose $H$ is not abelian. Then the set
$$
B=\{x: F^x\not\simeq \C\}
$$
has $\mu_\tau(B)\geq\frac 1 2$.
\end{lemma}

\begin{proof}
Let $g,h\in H$ be non-commuting elements of $S(H)$. Since $gh\neq hg$, we have that $u_{gh}$ and $u_{hg}$ are orthonormal vectors w.r.t.~$\langle\cdot,\cdot\rangle_\tau$, and so $\|u_{gh} - u_{hg} \|_\tau^2 =  \|u_{gh}\|_\tau^2+\|u_{hg}\|_\tau^2=2$. Thus
\begin{equation}\label{eq.intt}
2=\|u_{gh}-u_{hg}\|^2=\int_{B^c} \|u_{gh}^x-u_{hg}^x\|^2d\mu_\tau+\int_{B} \|u_{gh}^x-u_{hg}^x\|^2d\mu_\tau
\end{equation}
Since for $x\in B^c$ the factor $F^x$ is commutative (being isomorphic to $\C$), we have
$$
u_{gh}^x=u_g^xu_h^x=u_h^x u_g^x=u_{hg}^x.
$$
Thus the first integral above is 0. On the other hand, since $u_{gh}^x$ and $u_{hg}^x$ are unit vectors in $\langle\cdot,\cdot\rangle_{\tau_x}$, we have $\|u_{gh}^x-u_{hg}^x\|\leq 2$, and so the second integral above can be at most $4\mu_\tau(B)$. Thus \cref{eq.intt} gives $2\leq 4\mu_\tau(B)$, so $\mu_\tau(B)\geq\frac 1 2$ follows.
\end{proof}

Before the next proof, we remark that when $H_0,H_1<G$ are commuting subgroups, then $H_0H_1=\{h_0h_1: h_0\in H_0\wedge h_1\in H_1\}$ is a subgroup of $G$.

\begin{Lem}\label{l.tensor}
Suppose that $H_0,H_1 < G$ are commuting subgroups. Let $S(H_i)=\int F^x_i d\mu_i(x)$ and $S(H_0H_1)=\int F^x d\mu(x)$ be the factor decompositions.
Let $p_i \in S(H_i)$, $i\in\{0,1\}$, be central projections. Then

\begin{enumerate}
\item $p_0p_1$ is a central projection in $S(H_0H_1)$.
\item If there are $n_0,n_1\in\N$ such that $F^x_i$ contains a unital copy of $M_{n_i}(\C)$ for almost all $x\in\supp(p_i)$, then for almost all $x\in\supp(p_0p_1)$ the factor $F^x$ contains a copy of $M_{n_0n_1}(\C)$.
\item If in (2) we have $F^x_i\simeq M_{n_i}(\C)$ for almost all $x\in\supp(p_i)$, then  
$$
(\forall^\mu x\in\supp(p_0p_1))\ F^x\simeq M_{n_0n_1}(\C).
$$
\end{enumerate}
\end{Lem}
\begin{proof}
(1) Since $p_0$ and $p_1$ commute, it is clear that $p_0p_1$ is a projection, and since each $p_i$ commutes with all $u_g$ when $g\in H_0\cup H_1$, it follows that $p_0p_1$ is a central projection in $S(H_0H_1)$.

\medskip

(2) We may assume that $p_0p_1\neq 0$, otherwise there is nothing to show. By the Jankov-von Neumann uniformization Theorem \cite[Theorem 18.1]{Kechris95}, we can find measurable functions $x\mapsto e^{x,i}_{jk}$, where $x\in\supp(p_i)$, such that $(e^{x,i}_{jk})_{1\leq j,k\leq n_i}$ is a system of (non-zero) matrix units (see \cite[p.~159]{Blackadar06} or \cite[pp.~109--110]{RoLaLa00}) in $F^x_i$. Clearly
$$
e^i_{jk}=\int e^{x,i}_{jk}d\mu_i(x)
$$
is then a set of $n_i\times n_i$ matrix units in $p_i S(H_i)$.

\medskip

{\sc Claim.} Let $p\in Z(S(H_0H_1))$ be a central projection with $0\neq p\leq p_0p_1$. Then, for $1\leq j,k,a\leq n_0$ and $1\leq l,m,b\leq n_1$,
\begin{enumerate}[(a)]
\item $p e^0_{jk}e^1_{lm}\neq 0$;
\item $(p e^0_{ja}e^1_{lb})(p e^0_{ak}e^1_{bm})=pe^0_{jk}e^1_{lm}$;
\item $(p e^0_{jk}e^1_{lm})^*=p e^0_{kj}e^1_{ml}$.
\end{enumerate}
Thus
$$
\{p e^0_{jk}e^1_{lm}: 1\leq j,k\leq n_0\wedge 1\leq l,m\leq n_1\}
$$ 
forms a system of $n_0n_1\times n_0n_1$ matrix units.

\medskip

{\it Proof of Claim.} (b) and (c) follow by direct calculation from $(e^0_{jk})_{1\leq j,k\leq n_0}$ and $(e^1_{lm})_{1\leq l,m\leq n_1}$ being matrix units. To see (1) it is enough to observe that
$$
p_{km}=(pe^0_{jk}e^1_{lm})^*(pe^0_{jk}e^1_{lm})=pe^0_{kk}e^1_{mm}\neq 0 \, .
$$
Since all the projections $p_{km}$ are Murray-von Neumann equivalent by (b) and (c) above, it follows that if one of them is equal to $0$ then all of them must be equal to $0$. Since
$$
\sum_{k,m} p_{km}=\sum_{k,m} pe^0_{kk}e^1_{mm}=p \, ,
$$
not all $p_{km}$ can be zero.
\hfill $_{\text{Claim.}}\dashv$

\medskip

The previous claim clearly gives that $F^x$ contains a system of non-zero $n_0n_1\times n_0n_1$ matrix units for almost every $x\in\supp(p_0p_1)$.

\medskip

(3) We will not need this part of the lemma. All the same, to see that $F^x\simeq M_{n_0n_1}(\C)$ under the additional assumptions of (3), we just need to show that the corner $pS(H_0H_1)$ is generated by the set
$$
\{qe^0_{jk}e^1_{lm}: 1\leq j,k\leq n_0\wedge 1\leq l,m\leq n_1\wedge q\in Z(S(H_0H_1))\wedge q\leq p\}.
$$
This follows since, by the definition of the $e^i_{jk}$, the corner $pS(H_i)$ is generated by
$$
\{pe^i_{jk}: 0\leq i\leq 1\wedge 1\leq j,k\leq n_i\}\cup \{q\in Z(S(H_i)): q\leq p\},
$$
and it is clear that $Z(S(H_i))\subseteq Z(S(H_0H_1))$, and $pS(H_0H_1)$ is generated by $pS(H_0)\cup pS(H_1)$. \end{proof}

The next two lemmata show how \Cref{l.nonab} and \ref{l.tensor} together can be used to achieve dimension growth in the factor decomposition in the presence of a sequence of commuting non-abelian subgroups.

\begin{lemma}\label{l.estimate}
Let $1>\varepsilon>0$, $k>1$. Let $G$ be a group which contains a sequence of commuting subgroups $(G_i)_{i\in\N}$, and suppose the factor decompositions $S(G_i)=\int_{X_i} F^x_id\mu_i(x)$ satisfy that the sets
$$
B_i=\{x: \dim(F^x_i)\geq k\}
$$
have $\mu_i(B_i)>\frac 1 2  -\varepsilon$.
Then there is $N\in\N$ such that the factor decomposition $S(G_0\cdots G_N)=\int F^xd\mu(x)$ satisfies
$$
\mu_\tau(\{x: \dim(F^x)\geq k^2\})>\frac 1 2 -\varepsilon.
$$
\end{lemma}
\begin{proof}
Let $G_\infty=\langle\bigcup_{i\in\N} G_i\rangle)$, i.e.,~$G_{\infty}$ is the group generated by $\bigcup_{i\in\N} G_i$. Notice that the center of $S(G_i)$ and of $S(G_1\cdots G_N)$ are contained in the center of $S(G_\infty)$, so we may consider all the sets $B_i$ as belonging to the same probability space, namely the probability space corresponding to $Z(S(G_\infty)$), with the measure $\mu_\tau$ derived from the trace. By the previous lemma, it holds that for any $N\in\N$ the factor decomposition $S(G_1\cdots G_N)=\int F^xd\mu_\tau(x)$ satisfies
$$
\{x: \dim(F^x)>k\}\supseteq\bigcup_{i\leq N} B_i.
$$
As $N\mapsto \mu_\tau(\bigcup_{i<N} B_i)$ is bounded by 1, there must be some $N$ such that $\mu_\tau(B_{N}\setminus \bigcup_{i< N} B_i)<\mu_\tau(B_N)-\frac 1 2+\varepsilon$. Then, by the previous lemma applied to the groups $G_1\cdots G_{N-1}$ and $G_{N}$ we have for the factor decomposition $S(G_1\cdots G_{N})=\int F^xd\mu_\tau(x)$ that
$$
\{x: \dim(F^x)\geq k^2\}\supseteq B_{N}\cap \bigcup_{i<N} B_i,
$$
and by the choice of $N$, the set on the right hand side has measure greater than $\mu_\tau(B_n)-(\mu_\tau(B_n)-\frac 1 2+\varepsilon)=\frac 1 2-\varepsilon$.
\end{proof}

\begin{lemma}\label{l.dimgrowth}
Let $1>\varepsilon>0$, $k \geq 1$. Let $G$ be a group which contains a sequence of commuting non-abelian subgroups $(G_i)_{i\in\N}$. Then there is $N\in\N$ such that for the direct integral decomposition $S(G_1\cdots G_N)=\int F^xd\mu(x)$ it holds that
$$
\mu(\{x: \dim(F^x)\geq 2^{2^{k-1}}\})>\frac 1 2-\varepsilon.
$$
\end{lemma}

\begin{proof}
By induction on $k$. The case $k=1$ follows from \Cref{l.nonab}. The inductive step follows from the previous lemma.
\end{proof}

\begin{proof}[Proof of \Cref{t.hard} (A)]
Let $G_\infty$ be as in the proof of the previous lemma, and let $S(G_\infty)=\int F^x d\mu_\tau(x)$ be the direct integral decomposition, with $\mu_\tau$ the measure derived from the trace.

\medskip

{\sc Claim.} $\mu_\tau(\{x: \dim(F^x)=\infty\})\geq \frac 1 2$, and so $S(G_\infty)$ is not type I.

\medskip

\noindent {\it Proof of Claim.}
Fix $\varepsilon>0$. It suffices to show that $\mu_\tau(\{x: \dim(F^x)>k\})> \frac 1 2-\varepsilon$ for all $k\in\N$.

Using the previous lemma, find $N\in\N$ such that for the direct integral decomposition $S(G_1\cdots G_N)=\int E^xd\mu_\tau(x)$ we have
$$
\mu_\tau(\{x: \dim(E^x)>k\})>\frac 1 2-\varepsilon.
$$
Let $G_{N+1,\infty}=\langle\bigcup_{i>N} G_i\rangle$. Since $G_\infty=(G_1\cdots G_N)G_{N+1,\infty}$, it follows by \Cref{l.tensor} that
$$
\mu_\tau(\{x: \dim(F^x)>k\})>\frac 1 2 -\varepsilon.
$$
\hfill \hfill $_{\text{Claim.}}\dashv$

\noindent As $L(G)\supseteq S(G_\infty)$, it follows that $L(G)$ is not type I.
\end{proof}

{\sc Remark.} The reader may note that the sequence of Lemmata \ref{l.nonab}--\ref{l.dimgrowth} can be used to give an explicit embedding of the hyperfinite $\II_1$ factor $R$ into a corner of $L(G)$ under the assumptions of \Cref{t.hard} (A).

\subsection{Proof of part (B) of \Cref{t.hard}}

We continue to let $G$ denote a countable discrete group. The following lemma gives the sequence of subgroups needed for \Cref{t.hard} (B).

\begin{lemma}
Suppose $\gf$ is not abelian-by-finite. Then there are sequences $(g_i)_{i\in\N}$ and $(h_i)_{i\in\N}$ of elements of $\gf$, and a sequence of subgroups $(G_i)_{i \in \omega}$ of $\gf$, having the following properties:

\begin{enumerate}
\item $g_i$ and $h_i$ do not commute.
\item $G_i=\langle [g_i]_{\gf}, [h_i]_{\gf}\rangle$. In particular, $G_i$ is non-abelian by (1).
\item $G_i$ and $G_j$ are commuting subgroups of $G$ whenever $i\neq j$ (i.e.,~all elements of $G_i$ commute with all elements of $G_j$ whenever $i\neq j$).
\end{enumerate}
\end{lemma}

\begin{proof} We will define the sequences $(g_i)_{i\in\N}$ and $(h_i)_{i\in\N}$ recursively and see that (1), (2) and (3) hold.

First pick $g_1,h_1\in \gf$ to be any two non-commuting elements of $\gf$. Suppose then that $g_1,\ldots, g_k$ and $h_1,\ldots, h_k$ have been defined, and that (1), (2) and (3) are satisfied. Let
$$
K=\bigcup\{[g_i]_{\gf},[h_i]_{\gf}: 1\leq i\leq k\},
$$
and note that $K$ is finite and invariant under conjugation in $\gf$. For each $g\in \gf$, let $\gamma_g:\gf \to \gf: h\mapsto ghg^{-1}$. Note that $\gamma_g(K)=K$, and each $\gamma_g$ is completely determined by its values on $K$. Thus $\psi: g\mapsto\gamma_g\restrict K$ is a homomorphism from $\gf$ to $S_K$, the group of permutations of $K$. Since $K$ is finite, $\ker(\psi)$ has finite index in $\gf$, and since $aba^{-1}=b$ for any $a\in\ker(\psi)$ and $b\in K$, it follows that the elements of $\ker(\psi)$ commute with all elements in the subgroups $G_1,\ldots, G_k$. Since $\gf$ is not abelian-by-finite, $\ker(\psi)$ is not abelian, and so we can pick non-commuting elements $g_{k+1},h_{k+1}\in \ker(\psi)$. Since $\ker(\psi)$ is normal in $\gf$ we have $[g_{k+1}]_{\gf},[h_{k+1}]_{\gf}\subseteq \ker(\psi)$, and so $G_{k+1}=\langle [g_{k+1}]_{\gf},[h_{k+1}]_{\gf}\rangle$ commutes with all the subgroups $G_1,\ldots, G_k$.\end{proof}

\begin{proof}[Proof of \Cref{t.hard} {\rm (B)}]
If $[G:\gf]<\infty$ and $G$ is not abelian-by-finite, then $\gf$ is not abelian-by-finite. By the previous lemma, $\gf$, and therefore $G$, contains a sequence of non-abelian pairwise commuting subgroups, as required.
\end{proof}

\bibliographystyle{amsalpha}
\bibliography{Thoma-new_proof}
\end{document}